\documentclass{amsart}
\usepackage{amsfonts,amssymb,amscd,amsmath,enumerate,verbatim,calc}
\usepackage[all]{xy}
\usepackage{euscript}

\newcommand{\CM}{Cohen-Macaulay}

\newcommand{\n}{\mathfrak{n} }
\newcommand{\m}{\mathfrak{m} }

\newcommand{\q}{\mathfrak{q} }

\newcommand{\Pb}{\mathbb{P} }
\newcommand{\Z}{\mathbb{Z} }

\newcommand{\Vc}{\mathcal{V}}

\newcommand{\Ic}{\mathcal{I} }
\newcommand{\Sc}{\mathcal{S} }
\newcommand{\Tc}{\mathcal{T} }

\newcommand{\rt}{\rightarrow}

\newcommand{\Om}{\Omega}

\newcommand{\wt}{\widetilde }

\newcommand{\image}{\operatorname{image}}
\newcommand{\Mor}{\operatorname{Mor}}

\newcommand{\ann}{\operatorname{ann}}
\newcommand{\cx}{\operatorname{cx}}
\newcommand{\ecx}{\operatorname{ecx}}

\newcommand{\topv}{\operatorname{top}}

\newcommand{\CMS}{\operatorname{\underline{CM}}}

\newcommand{\CMa}{\operatorname{CM}}

\newcommand{\projdim}{\operatorname{projdim}}
\newcommand{\Hom}{\operatorname{Hom}}

\newcommand{\sHom}{\operatorname{\underline{Hom}}}
\newcommand{\Ext}{\operatorname{Ext}}
\newcommand{\Tor}{\operatorname{Tor}}

\theoremstyle{plain}

\newtheorem{theorem}{Theorem}[section]
\newtheorem{corollary}[theorem]{Corollary}
\newtheorem{lemma}[theorem]{Lemma}
\newtheorem{proposition}[theorem]{Proposition}

\newtheorem{conjecture}[theorem]{Conjecture}

\theoremstyle{definition}
\newtheorem{definition}[theorem]{Definition}
\newtheorem{s}[theorem]{}
\newtheorem{remark}[theorem]{Remark}

\theoremstyle{remark}

\begin{document}

\title[Support varieties]{Support Varieties and cohomology of Verdier quotients of stable category of complete intersection rings}
\author{Tony~J.~Puthenpurakal}
\date{\today}
\address{Department of Mathematics, IIT Bombay, Powai, Mumbai 400 076}

\email{tputhen@math.iitb.ac.in}
\subjclass{Primary  13C14,  13D09 ; Secondary 13C60, 13D02}
\keywords{stable category of Gorenstein rings, complete intersections, support varieties, Hensel rings  periodic complexes, Auslander-Reiten triangles}

 \begin{abstract}
Let $(A,\m)$ be a complete intersection with $k = A/\m$ algebraically closed. Let $\CMS(A)$ be the stable category of maximal \CM \ $A$-modules. For a large class of thick subcategories $\Sc$ of $\CMS(A)$ we show that there is a theory of support varieties for the Verdier quotient $\Tc = \CMS(A)/\Sc$. As an application we show that the analogous
version of Auslander-Reiten conjecture, Murthys result, Avramov-Buchweitz result on symmetry of vanishing of cohomology holds for $\Tc$
\end{abstract}
 \maketitle
\section{introduction}
Quillen's geometric methods to study cohomology of finite groups is an important contribution in modular representation theory, see \cite{Q}. The  techniques
involved have been generalized and extended to representations of various
Hopf algebras, for instance see  \cite{FS}. In commutative algebra Avramov and Buchweitz introduced the notion of support varieties of a pair of modules over local complete intersections and as an application proved the symmetry of vanishing of Ext over such rings; see \cite{avr-b}. In \cite{B2}, \cite{BO} the notion of support varieties was extended to certain class of triangulated categories

Let $(A,\m)$ be a  commutative Gorenstein local ring with residue field $k$. Let $\CMa(A)$ denote the full subcategory of  maximal \CM \ (= MCM) $A$-modules  and let $\CMS(A)$ denote the stable category of MCM $A$-modules.
It is well-known that $\CMS(A)$ is  a triangulated category with translation functor $\Omega^{-1}$,  (see \cite{Bu}; cf. \ref{T-struc}).

We use Neeman's book \cite{N} for notation on triangulated categories. However we will assume that if $\mathcal{C}$ is a triangulated category then $\Hom_\mathcal{C}(X, Y)$ is a set for any objects $X, Y$ of $\mathcal{C}$.

\s \label{hypothesis}  For the rest of the paper let us assume that $(A,\m)$  is a complete complete intersection ring of dimension $d$ and codimension $c$.
 Assume $k = A/\m$ is algebraically closed.  Some of our results are applicable  more generally. However for simplicity
 we will make this hypothesis throughout this paper.

 There is a theory of support varieties for modules over $A$. Essentially  for every  finitely generated module $E$ over  $A$ an
 algebraic cone $\Vc(E)$ in  $k^{c}$ is attached, see \cite[6.2]{A}.
 Conversely it is known that if $V$ is an algebraic cone in $k^c$ then there exists a finitely generated module $E$ with $\Vc(E) = V$,
 see \cite[2.3]{B}. It is known that $\Vc(\Om^n(E)) = \Vc(E)$ for any $n \geq 0$. Thus we can assume $E$ is maximal \CM.

 \begin{remark}
   It is easier to work with varieties in $\Pb^{c-1}$ than with algebraic cones in $k^c$. So we let $\Vc^*(E)$ be the algebraic set in $\Pb^{c-1}$ corresponding to $\Vc(E)$.
 \end{remark}
 We now give the class of thick subcategories of $\CMS(A)$ that are of interest to us.
 \begin{definition}\label{defnX}
   Let $X$ be an algebraic set in $\Pb^{c-1}$. Set
   \[
   \Sc_X = \{ M \mid M \ \text{is MCM and  } \  \Vc^*(M) \subseteq X \}
   \]
   It is easy to verify that $\Sc_X$ is a thick subcategory of $\CMS(A)$; see Lemma \ref{defX-thick}. We will be interested in the Verdier quotient
   \[
   \Tc_X = \CMS(A)/\Sc_X.
   \]
   Our definition of support variety for objects in $\Tc_X$ is simply
   \[
   \Vc_X(M) = \Vc^*(M) \setminus X.
   \]
   In Corollary \ref{def-supp-X}  we show that this is a well-defined notion in $\Tc_X$.
 \end{definition}

 Although our definition of support varieties in $\Tc_X$ is very simple,  it has significant consequences. We now describe our:

 \textit{Applications:}\\
 (1) Auslander and Reiten conjectured that if $R$ is an Artin ring, $\Lambda $ is an Artin $R$-algebra and $M$ is a finitely generated $\Gamma$-module then
 \[
\Ext^i_\Lambda(M, M \oplus \Lambda) = 0 \quad \text{for all} \ i \geq 1 \implies M \ \text{is projective}.
 \]
 We note that this conjecture makes sense for any ring $\Lambda$. In \cite[1.9]{ADS} it is shown that AR-conjecture holds for complete  intersection rings. Also see
 \cite[Main Theorem]{HL} and \cite[Corollary 4]{Ar} where it is shown that AR-conjecture holds for normal local Gorenstein rings.
 The generalized Auslander Reiten (GAR) conjecture states that if $\Lambda$ is any ring then
 \[
 \Ext^i_\Lambda(M, M \oplus \Lambda) = 0 \quad \text{for all} \ i \geq m \implies \projdim_\Lambda M < m.
 \]
 GAR has been verified for complete intersections, see \cite[4.2]{avr-b}. In this case we just require $\Ext^{2i}_R(M, M) = 0$ for some $i \geq 1$.
 We note that if $R$ is Gorenstein local and $M, N$ are MCM $R$-modules then $\Ext^{i}_R(M, N) = \sHom_R(M, \Om^{-i}(N))$.
 We say a triangulated category $\Tc$ with shift operator $\Sigma$ satisfies \emph{generalized Auslander-Reiten property (GAR) }if for $U \in \Tc$
 \[
 \Hom_\Tc(U, \Sigma^n U) = 0 \ \text{for all} \ n \gg 0 \implies U = 0.
 \]
 Our first application is
 \begin{theorem}
   \label{GAR}
(with hypotheses as in \ref{hypothesis} and notation as in \ref{defnX}.) The triangulated category $\Tc_X$ satisfies GAR for any algebraic set $X \subseteq \Pb^{c-1}$.
 \end{theorem}

 (2) Murthy, \cite[1.6]{M} proved that if $A$ is a complete intersection  of codimension $c$ and for some $m > 0$
  $$\Tor^A_i(M, N) = 0 \quad \text{ for} \  i = m, m+1,\cdots, m+c; $$
  then $\Tor^A_i(M, N) = 0$ for all $i \geq m$.
  The corresponding property with $\Ext^*(-, -)$ is also true.
  We say a triangulated category $\Tc$ with shift operator $\Sigma$ satisfies \emph{Murthy's property }with order $r \geq 1$  if for some $m$ and  $U, V \in \Tc$
  \begin{align*}
    \Hom_\Tc(U, \Sigma^n V) &= 0 \ \text{for } \ n = m, m+1, \cdots. m+r   \\
    &\implies  \Hom_\Tc(U, \Sigma^n V) = 0 \ \text{for all }  \  n \geq m.
  \end{align*}
  Clearly $\CMS(A)$ satisfies Murthy's property with order $c$ if $A$ is a local complete intersection of codimension $c$. See \cite{B2} for some examples of triangulated categories satisfying Murthy's property.
 Our second application is
 \begin{theorem}
   \label{murthy}
(with hypotheses as in \ref{hypothesis} and notation as in \ref{defnX}.) The triangulated category $\Tc_X$ satisfies Murhty's property with order $c$  for any algebraic set $X \subseteq \Pb^{c-1}$.
 \end{theorem}

 (3) A spectacular application of Avramov and Buchwetiz's definition of support variety of a pair of modules over a complete intersection $A$ is the following:
 \[
 \Ext^{i}_A(M, N) = 0 \ \text{for all} \ i \gg 0 \quad \text{if and only if} \quad  \Ext^{i}_A(N, M) = 0 \ \text{for all} \ i \gg 0
 \]
 We say a triangulated category $\Tc$ with shift operator $\Sigma$ satisfies \emph{symmetry in vanishing of cohomology }if for $U, V \in \Tc$
 \[
 \Hom_\Tc(U, \Sigma^n V) = 0 \ \text{for all} \ n \gg 0 \implies \Hom_\Tc(V, \Sigma^n U) = 0 \ \text{for all} \ n \gg 0.
 \]
  Clearly $\CMS(A)$ satisfies symmetry in vanishing of cohomology if $A$ is a local complete intersection of codimension $c$. See \cite{BO} for some examples of triangulated categories satisfying symmetry  in vanishing of cohomology.

  Before we state our next result we need the following:
  \begin{definition}\label{ess}
  We say a module $M$ is essentially disjoint from $X$ if $M = M_1\oplus M_2$ where $\Vc^*(M_1) \subseteq X$ and $\Vc^*(M_2) \subseteq \Pb^{c-1} \setminus X$.
  \end{definition}
  \begin{remark}
    (1) Note $M \cong M_2$ in $\Tc_X$.

    (2) In \ref{fart} we show that if $M \cong N$ in $\Tc_X$ and $M$ is essentially disjoint from $X$ then so is $N$.
  \end{remark}
  Our third application is
 \begin{theorem}
   \label{sym}
(with hypotheses as in \ref{hypothesis} and notation as in \ref{defnX}.) Let $X$ be any algebraic set in $\Pb^{c-1}(k)$. Let $M, N$ be any two MCM $A$-modules.
Consider the following two statements:
\begin{enumerate}[\rm (1)]
  \item $\Hom_{\Tc_X}(M, \Om^{-n}(N)) = 0 \ \text{for all} \ n \gg 0$.
  \item $\Vc_X(M) \cap \Vc_X(N) = \emptyset$.
\end{enumerate}
Then $(1) \implies (2)$.
If $M$ or $N$ is essentially disjoint from $X$ then $(2) \implies (1)$.
 \end{theorem}
 We make the following:
 \begin{conjecture}
   \label{sym-conj}(with hypotheses as in Theorem \ref{sym}) The assertions $(2) \implies (1)$.
hold in general.
 \end{conjecture}

 \begin{remark}
   The main reason we are unable to prove Conjecture \ref{sym-conj} is that we do not have a good notion of support variety for  a pair of objects $M, N$ in $\Tc_X$. Philosophically it should be $\Vc_X(M) \cap \Vc_X(N)$ but we do not have a cohomological criterion for support varieties of the pair $M, N$.
 \end{remark}

\s \label{defnC} We give another class of Verdier-quotient's of $\CMS(A)$ for which GAR and Muthy's property holds. Let $\cx_A M$ denote the complexity of a module $M$.
 For $i =   1,\ldots, c-1$ let
 $$\CMS_{\leq i}(A) = \{ M \mid \text{ $M$ is MCM and } \ \cx_A M \leq i \}.$$
 Then it is easy to check that  $\CMS_{\leq i}(A)$ is a thick subcategory of $\CMS(A)$. Set
 $$ \Tc_i = \CMS(A)/\CMS_{\leq i}(A). $$
 If $Y$ is a variety in $\Pb^{c-1}$, write
 $$ Y = \bigcup_{j = 1}^{m} Y_j \quad \text{where} \ Y_j \ \text{is irreducible}.$$
 Assume $\dim Y_j = \dim Y$ for $1\leq j \leq r$ and $\dim Y_j < \dim Y$ for $j > r$.
 Set
 \[
 \topv(Y) = \bigcup_{i = 1}^{r} Y_i.
 \]
 It is easy to prove that $\topv(Y)$ is an invariant of $Y$.
 We define support variety of an object $M$ in $\Tc_i$ as follows
 \[
 \Vc_i(M)  = \begin{cases}
               \topv(\Vc^*(M)), & \mbox{if } \ \cx_A M > i; \\
               \emptyset, & \mbox{otherwise}.
             \end{cases}
 \]
 In Corollary \ref{def-supp-cmi} we show that this is a well-defined notion in $\Tc_i$.
 We show
  \begin{theorem}
   \label{GARC}
(with hypotheses as in \ref{hypothesis} and notation as in \ref{defnC}.)
 For $i = 1,\ldots, c-1$
 the triangulated category $\Tc_i$ satisfies GAR.
 \end{theorem}

 Next we show:
 \begin{theorem}
   \label{murthyC}
(with hypotheses as in \ref{hypothesis} and notation as in \ref{defnC}.)
 For $i = 1,\ldots, c-1$
 the triangulated category $\Tc_i$ satisfies Murthy's property with order $c-i$.
 \end{theorem}

 Huneke and Wiegand, \cite[1.9]{HW} showed that if $A$ is a hypersurface ring and $M,N$ are MCM $A$-modules then
 for some $i\geq 1$ we have
 \[
 \Tor^A_i(M, N) = \Tor^A_{i+1}(M, N) = 0 \quad \text{ then either $M$ or $N$ is free.}
 \]
 Similar result holds for $\Ext(-,-)$.
 Note that any MCM module over a hypersurface is two periodic, i.e., $\Om^2(M) = M$
for any MCM $A$-module $M$ with no free summands. We say a triangulated category $\Tc$ is two periodic if $\Sigma^2(U) = U$ for any object $U$ in $\Tc$. We say a two periodic triangulated category
$\Tc$ has the Huneke-Wiegand property if
\[
\Hom_\Tc(U,V) = \Hom_\Tc(U, \Sigma(V)) = 0 \quad \implies U = 0 \ \text{or} \ V = 0.
\]
In an earlier paper we showed that $\Tc_{c-1}$ is $2$-periodic, see \cite[3.1]{P}
In this paper we show:
\begin{theorem}
   \label{HW}
(with hypotheses as in \ref{hypothesis} and notation as in \ref{defnC}.)
 The triangulated category $\Tc_{c-1}$ has the Huneke-Wiegand property.
 \end{theorem}
\section{Preliminaries}
In this paper all rings will be Noetherian local. All  modules considered are \textit{finitely} generated unless otherwise stated.  
However note that Home-sets in Verdier quotients need not be finitely generated.

Let $(A,\m)$ be a local ring and let $k = A/\m$ be its residue field. Let $\dim A = d$. If $M$ is an $A$-module then $\mu(M) = \dim_{k} M/\m M$ is the number of a minimal generating set of $M$. Also let $\ell(M)$ denote its length.  In this section we discuss a few preliminary results that we need.

\s \label{P} Let $M$ be an $A$-module. For $i \geq 0$  let $\beta_i(M) = \dim_k \Tor^A_i(M, k)$ be its $i^{th}$ \emph{betti}-number.  Let $P_M(z) = \sum_{n \geq 0}\beta_n(M)z^n$, the \emph{Poincare series} of $M$.
 Set
 \[
 \cx(M) = \inf \{ r \mid \limsup \frac{\beta_n(M)}{n^{r-1}}  < \infty \}  \quad \text{and}
 \]
 It is possible that $\cx(M) = \infty$, see \cite[4.2.2]{A-Inf}.

 \s It can be shown that for any $A$-module $M$ we have
 \[
 \cx(M) \leq \cx(k) \quad \text{see \cite[4.2.4]{A-Inf}.}
 \]

 \s \label{ext-ci} If $A$  is a complete intersection of co-dimension  $c$ then for any $A$-module $M$ we have $\cx(M) \leq c$. Furthermore for each $i = 0,\ldots,c$ there exists an $A$-module $M_i$ with $\cx(M_i) = i$. Also note that $\cx(k) = c$. \cite[8.1.1(2)]{A-Inf}.

\s \emph{The stable category of a Gorenstein local ring:}\\
Let $(A,\m)$ be a  commutative Gorenstein local ring with residue field $k$. Let $\CMa(A)$ denote the full subcategory of  maximal \CM \ (= MCM) $A$-modules  and let $\CMS(A)$ denote the stable category of MCM $A$-modules. Recall that objects in $\CMS(A)$ are same as objects in $\CMa(A)$. However the set of morphisms $\underline{\Hom_A}(M,N)$ between $M$ and $N$ is  $= \Hom_A(M,N)/P(M,N)$ where $P(M,N)$ is the set of $A$-linear maps from $M$ to $N$ which factor through a  free module. It is well-known that $\CMS(A)$ is  a triangulated category with translation functor $\Omega^{-1}$,  (see \cite[4.7]{Bu}; cf. \ref{T-struc}).  Here $\Omega(M)$ denotes the syzygy of $M$ and $\Omega^{-1}(M)$ denotes the co-syzygy of $M$. Also recall that an object $M$ is zero in $\CMS(A)$ if and only if it is free  considered as an $A$-module. Furthermore $M \cong N$ in $\CMS(A)$ if and only
if there exists finitely generated free modules $F,G$ with $M\oplus F \cong N \oplus G$ as $A$-modules.

\s\label{T-struc} \emph{Triangulated category structure on $\CMS(A)$}. \\
We first describe
basic exact triangle. Let $f \colon M \rt N$ be a morphism in $\CMa(A)$. Note we have an exact sequence $0 \rt M \xrightarrow{i} Q \rt \Om^{-1}(M) \rt 0$, with $Q$-free. Let $C(f)$ be the pushout of $f$ and $i$. Thus we have a commutative diagram with exact rows
\[
  \xymatrix
{
 0
 \ar@{->}[r]
  & M
\ar@{->}[r]^{i}
\ar@{->}[d]^{f}
 & Q
\ar@{->}[r]^{p}
\ar@{->}[d]
& \Om^{-1}(M)
\ar@{->}[r]
\ar@{->}[d]^{j}
&0
\\
 0
 \ar@{->}[r]
  &N
\ar@{->}[r]^{i^\prime}
 & C(f)
\ar@{->}[r]^{p^\prime}
& \Om^{-1}(M)
    \ar@{->}[r]
    &0
\
 }
\]
Here $j$ is the identity map on $\Om^{-1}(M)$.
As $N, \Om^{-1}(M) \in \CMa(A)$  it follows that $C(f) \in \CMa(A)$.
Then the projection of the sequence
$$ M\xrightarrow{f} N \xrightarrow{i^\prime} C(f) \xrightarrow{-p^\prime} \Omega^{-1}(M)$$
in $\CMS(A)$  is a basic exact triangle. Exact triangles in $\CMS(A)$ are triangles isomorphic to a basic exact triangle.

\s By construction of exact triangles in $\CMS(A)$ it follows that if
\[
M \rt N \rt L \rt \Om^{-1}(M)
\]
is an exact triangle in $\CMS(A)$ then we have a short exact sequence
\[
0 \rt N \rt E \rt \Om^{-1}(M) \rt 0,
\]
where $E \cong L$ in $\CMS(A)$.

\s \textit{Support varieties of modules over local complete intersections:}\\
 This is relatively simple in our case since $A$ is complete with algebraically closed residue field.

\s Let $A = Q/(\mathbf{u})$ where $(Q,\n)$ is a complete regular local ring and \\ $\mathbf{u} = u_1,\ldots, u_c \in \n^2$ is   a regular sequence.
We need the notion of cohomological operators over a  complete intersection ring.

The \emph{Eisenbud operators}, \cite{E}  are constructed as follows: \\
Let $\mathbb{F} \colon \cdots \rightarrow F_{i+2} \xrightarrow{\partial} F_{i+1} \xrightarrow{\partial} F_i \rightarrow \cdots$ be a complex of free
$A$-modules.

\emph{Step 1:} Choose a sequence of free $Q$-modules $\wt{F}_i$ and maps $\wt{\partial}$ between them:
\[
\wt{\mathbb{F}} \colon \cdots \rightarrow \wt{F}_{i+2} \xrightarrow{\wt{\partial}} \wt{F}_{i+1} \xrightarrow{\wt{\partial}} \wt{F}_i \rightarrow \cdots
\]
so that $\mathbb{F} = A\otimes\wt{\mathbb{F}}$.

\emph{Step 2:} Since $\wt{\partial}^2 \equiv 0 \ \text{modulo} \ (\mathbf{u})$, we may write  $\wt{\partial}^2  = \sum_{j= 1}^{c} u_j\wt{t}_j$ where
$\wt{t_j} \colon \wt{F}_i \rightarrow \wt{F}_{i-2}$ are linear maps for every $i$.

 \emph{Step 3:}
Define, for $j = 1,\ldots,c$ the map $t_j = t_j(Q, \mathbf{f},\mathbb{F}) \colon \mathbb{F} \rightarrow \mathbb{F}(-2)$ by $t_j = A\otimes\wt{t}_j$.

\s
The operators $t_1,\ldots,t_c$ are called \emph{Eisenbud's operator's} (associated to $\mathbf{u}$) .  It can be shown that
\begin{enumerate}
\item
$t_i$ are uniquely determined up to homotopy.
\item
$t_i, t_j$ commute up to homotopy.
\end{enumerate}
\s Let $R = A[t_1,\ldots,t_c]$ be a polynomial ring over $A$ with variables $t_1,\ldots,t_c$ of degree $2$. Let $M, N$ be  finitely generated $A$-modules. By considering a free resolution $\mathbb{F}$ of $M$ we get well defined maps
\[
t_j \colon \Ext^{n}_{A}(M,N) \rightarrow \Ext^{n+2}_{R}(M,N) \quad \ \text{for} \ 1 \leq j \leq c  \ \text{and all} \  n,
\]
which turn $\Ext_A^*(M,N) = \bigoplus_{i \geq 0} \Ext^i_A(M,N)$ into a module over $R$. Furthermore these structure depend  on $\mathbf{u}$, are natural in both module arguments and commute with the connecting maps induced by short exact sequences.

\s\label{supp}  Gulliksen, \cite[3.1]{Gull},  proved that
$\Ext_A^*(M,N) $ is a finitely generated $R$-module. We note that $\Ext^*(M,k)$ is a finitely generated graded module over $T = k[t_1,\ldots, t_c]$. Define
$\Vc^*(M) = Var(\ann_T(\Ext^*(M,k))$ in the projective space $\mathbb{P}^{c-1}$. We call
$\Vc^*(M)$ the support variety of a module $M$. \emph{Note that in \cite{avr-b} support varieties are defined as $Var(\ann_T(\Ext^*(M,k))$ in
$k^c$. However as the ideal involved is homogeneous we get a similar notion.}
Set $\Vc^*(M,N) = Var(\ann_T(\Ext^*(M,N)\otimes k))$ in $\Pb^{c-1}$. Then in \cite[5.6]{avr-b} it is proved that $\Vc^*(M, N) =\Vc^*(M)\cap \Vc^*(N)$.

\s Set $\dim \emptyset = -1$. For any $A$-module $M$ it is clear that $\dim \Vc^*(M) = \cx_A M - 1$.

\s \label{ses}[\cite[5.6]{avr-b}:]
If $0\rt M_1 \rt M_2 \rt M_3 \rt 0$ is a short exact sequence and $N$ is any $A$-module then for $\{ i, j, h \} = \{ 1, 2, 3 \}$ we have
\[
\Vc^*(M_i,N) \subseteq \Vc^*(M_j,N) \cup \Vc^*(M_h,N).
\]

\s\label{ind} Let $a \in \Pb^{c-1}$. Define
\[
\Ic_a = \left \{  M \mid M \ \text{is MCM  and }\ \Vc^*(M) = \{ a \}                    \right \}.
\]
By \cite[2.3]{B} there exists $N$ with $\Vc^*(N) = \{ a \}$. Then $\Om^d(N) \in \Ic_a$. Thus $\Ic_a \neq \emptyset$ for all $a \in \Pb^{c-1}$.
We call $\Ic_a$ the class of indicator MCM's for $a \in \Pb^{c-1}$. The utility of $\Ic_a$ will be apparent in the next section.
We will need the following result:
\begin{lemma}\label{ind-lem}
(with notation as in \ref{ind}). For each $a \in \Pb^{c-1}$ choose any  $N_a \in \Ic_a$. Then for an $A$-module $M$ the following conditions are equivalent:
\begin{enumerate}[\rm (i)]
  \item  $a \in \Vc^*(M)$.
  \item $\Ext^*_A(M, N_a) \neq 0$.
  \item $\Ext^*_A(N_a, M) \neq 0$.
\end{enumerate}
\end{lemma}
\begin{proof}
  The result follows since for any $b \in \Pb^{c-1}$ we have
  \[
  \Vc^*(M, N_b) = \Vc^*(N_b, M)  =  \Vc^*(M)\cap \Vc^*(N_b) = \Vc^*(M) \cap \{ b \}.
  \]
\end{proof}

\s Let $\Tc$ be a triangulated category and let $\Sc$ be a thick subcategory of $\Tc$. Let $\Mor_\Sc$ denote the collections of morphisms $f \colon X \rt Y$ such that the cone of $f$ is in $\Sc$. Recall the Verdier quotient $\Tc/\Sc$ is obtained by formally inverting all morphisms in $\Mor_\Sc$; see \cite[Chapter 2]{N}.
A morphism $\phi \colon X \rt Y$ can be written as a left fraction
\[
\xymatrix{
\
&U
\ar@{->}[dl]_{u}
\ar@{->}[dr]^{f}
 \\
X
\ar@{->}[rr]_{\phi = fu^{-1}}
&\
&Y
}
\]
with $u \in \Mor_\Sc$;
or as a right fraction
\[
\xymatrix{
\
&V
\ar@{<-}[dl]_{g}
\ar@{<-}[dr]^{v}
 \\
X
\ar@{->}[rr]_{\phi = v^{-1}g}
&\
&Y
}
\]
with $v \in \Mor_\Sc$.
\section{Support Varieties for some Verdier quotients}
In this section we show that in the Verdier quotients that we consider objects  have a well-defined notion of support variety.
Throughout our assumptions will be as in \ref{hypothesis}. We will also consider support varieties of modules over complete intersections as defined in \ref{supp}.

\s Let $X$ be an algebraic set in $\Pb^{c-1}$.
 Set
   \[
   \Sc_X = \{ M \mid M \ \text{is MCM and  } \  \Vc^*(M) \subseteq X \}
   \]
\begin{lemma}
  \label{defX-thick}
  $\Sc_X $ is a thick subcategory of $\CMS(A)$.
\end{lemma}
\begin{proof}
Clearly $\Sc_X$ is closed under isomorphisms.
  If $M \in \Sc_X $ then as $\Vc^*(\Om^{-1}(M)) = \Vc^*(M)$ we get $\Om^{-1}(M) \in \Sc_X$.
  Let
  \[
  M \rt N \rt L \rt \Om^{-1}(M)
  \]
  be an exact triangle in $\CMS(A)$ with $M, N \in \Sc_X$. We then have a short exact sequence
  \[
  0 \rt N \rt L' \rt \Om^{-1}(M) \rt 0
  \]
  with $L' \cong L$ in $\CMS(A)$. By \ref{ses} we get $\Vc^*(L') \subseteq X$. Note $\Vc^*(L) = \Vc^*(L')$.
  Thus $\Sc_X$ is a triangulated subcategory of $\CMS(A)$. Let $M \oplus N \in \Sc_X$.
  As
  \[
  \Vc^*(M\oplus N) = \Vc^*(M)\cup \Vc^*(N)
  \]
  we get $M, N \in \Sc_X$. Thus $\Sc_X$ is a thick subcategory of $\CMS(A)$.
\end{proof}

\s Let $\Tc_X = \CMS(A)/\Sc_X$ denote the Verdier quotient of $\CMS(A)$ by $\Sc_X$.

\begin{lemma}\label{mor-X}
(with hypotheses as in \ref{defX-thick}) Let $f \colon M \rt N$ be in $\Mor_{\Sc_X}$. Then
\[
\Vc^*(M)\setminus X = \Vc^*(N)\setminus X.
\]
\end{lemma}
\begin{proof}
  We have a triangle
  \[
  \Om(L) \rt M \rt N \rt L
  \]
  with $L \in \Sc_X$. Let $a \in \Pb^{c-1} \setminus X$ and let $U \in \Ic_a$. Then notice $\sHom_A(U, L) = \sHom_A(U, \Om(L)) = 0$.
  So we have an isomorphism:
  \[
  \sHom_A(U, M) \cong \sHom_A(U, N).
  \]
  Similarly we have
  an isomorphism:
  \[
  \sHom_A(\Om{U}, M) \cong \sHom_A(\Om{U}, N).
  \]
  The result follows from \ref{ind-lem}.
\end{proof}
As a consequence we have:
\begin{corollary}
  \label{def-supp-X} (with hypotheses as in \ref{defX-thick}) Let $\phi \colon M \rt N$ be an isomorphism in $\Tc_X$. Then
   \[
\Vc^*(M)\setminus X = \Vc^*(N)\setminus X.
\]
\end{corollary}
\begin{proof}
We write $\phi$ as a left fraction $fu^{-1}$
\[
\xymatrix{
\
&U
\ar@{->}[dl]_{u}
\ar@{->}[dr]^{f}
 \\
M
\ar@{->}[rr]_{\phi = fu^{-1}}
&\
&N
}
\]
with $u \in \Mor_{\Sc_X}$.
By Lemma \ref{mor-X}, we get
\[
\Vc^*(M)\setminus X = \Vc^*(U)\setminus X.
\]
As $\phi$ is an isomorphism we get that $f \in \Mor_{\Sc_X}$ also. Again by
Lemma \ref{mor-X}, we get
\[
\Vc^*(N)\setminus X = \Vc^*(U)\setminus X.
\]
The result follows.
\end{proof}

\begin{definition}
\label{supp-var-X}
Let $M \in \Tc_X$. Set
\[
\Vc_X(M) = \Vc^*(M)\setminus X.
\]
By Corollary \ref{def-supp-X} this is a well-defined invariant of $M$ in $\Tc_X$.
\end{definition}

\s\label{cmi} For $i = 1,\ldots, c-1$ let
 $$\CMS_{\leq i}(A) = \{ M \mid \text{ $M$ is MCM and } \ \cx_A M \leq i \}.$$
 Then by a proof similar to \ref{defX-thick} we get that  $\CMS_{\leq i}(A)$ is a thick subcategory of $\CMS(A)$. Set
 $$ \Tc_i = \CMS(A)/\CMS_{\leq i}(A). $$

 \s If $Y$ is a variety in $\Pb^{c-1}$, write
 $$ Y = \bigcup_{j = 1}^{m} Y_i \quad \text{where} \ Y_j \ \text{is irreducible}.$$
 Assume $\dim Y_j = \dim Y$ for $1\leq j \leq r$ and $\dim Y_j < \dim Y$ for $i > r$.
 Set
 \[
 \topv(Y) = \bigcup_{j = 1}^{r} Y_j.
 \]
 It is easy to prove that $\topv(Y)$ is an invariant of $Y$.
 \begin{lemma}\label{mor-cmi}
(with hypotheses as in \ref{cmi}) Let $f \colon M \rt N$ be in $\Mor_{\CMS_{\leq i}}$. Then
\[
\topv(\Vc^*(M)) = \topv(\Vc^*(N)).
\]
\end{lemma}
\begin{proof}
We write $\Vc^*(M)$ and $\Vc^*(N)$ as a union of irreducible subvarieties;
 $$ \Vc^*(M) = \bigcup_{j = 1}^{m} Y_j \quad \text{and} \quad \ \Vc^*(N) = \bigcup_{j = 1}^{m} Z_j  $$
 Assume $\dim Y_j = \dim \Vc^*(M) $ for $1\leq j \leq r$ and $\dim Y_j <  \Vc^*(M) \dim $ for $j > r$. Also assume
 $\dim Z_j = \dim \Vc^*(N) $ for $1\leq j \leq s$ and $\dim Z_j < \dim  \Vc^*(N)$ for $j > s$.

We have a triangle
  \[
  \Om(L) \rt M \rt N \rt L
  \]
  with $L \in \CMS_{\leq i}$.
Fix $j$ with $1\leq j \leq  r$. As $\dim \Vc^*(L) \leq i-1$ and $\dim Y_j  \geq i$ we get that $Y_j \setminus \Vc^*(L) \neq \emptyset$.
Let $a \in  Y_j \setminus \Vc^*(L)$ and let $U \in \Ic_a$. By an argument similar to proof of Lemma \ref{mor-X} we get that
$\Ext^*(U, N)\neq 0$. So $a \in \Vc^*(N)$. Thus
\[
Y_j \setminus \Vc^*(L) \subseteq \Vc^*(N)
\]
It follows that $Y_j  \subseteq \Vc^*(N)$. An elementary argument yields that $Y_j \subseteq \topv(\Vc^*(N))$.
Thus $ \topv(\Vc^*(M))\subseteq \topv(\Vc^*(N))$. An analogous argument yields  $ \topv(\Vc^*(N))\subseteq \topv(\Vc^*(M))$.
\end{proof}
As a consequence we have:
\begin{corollary}
  \label{def-supp-cmi} (with hypotheses as in \ref{cmi}) Let $\phi \colon M \rt N$ be an isomorphism in $\Tc_i$. Then
   \[
\topv(\Vc^*(M)) \cong \topv(\Vc^*(N)).
\]
\end{corollary}
The proof of the above Corollary uses \ref{mor-cmi} and is analogous to proof of Corollary \ref{def-supp-X}.

\section{Proof of Theorems \ref{sym} (partly) and \ref{GAR}}
In this section we give proof of part of Theorem \ref{sym}. As a consequence  we give a proof of Theorem \ref{GAR}. We restate the part of Theorem \ref{sym} we will prove:
\begin{theorem}
   \label{sym-main}
(with hypotheses as in \ref{hypothesis} and notation as in \ref{defnX}.) Let $X$ be any algebraic set in $\Pb^{c-1}(k)$ . Let $M, N$ be any two MCM $A$-modules.
Consider the following two statements:
\begin{enumerate}[\rm (1)]
  \item $\Hom_{\Tc_X}(M, \Om^{-n}(N)) = 0 \ \text{for all} \ n \gg 0$.
  \item $\Vc_X(M) \cap \Vc_X(N) = \emptyset$.
\end{enumerate}
Then $(1) \implies (2)$.
 \end{theorem}
 As a corollary we prove
 \begin{proof}[Proof of Theorem \ref{GAR}]
 By Theorem \ref{sym-main} we get $\Vc_X(M)  = \emptyset$. So $\Vc^*(M)\subseteq X$. This implies $M = 0$ in $\Tc_X$.
 \end{proof}
We now give
\begin{proof}[Proof of Theorem \ref{sym-main}]
Assume $\Hom_\Tc(M, \Om^{-n}(N)) = 0 \ \text{for all} \ n \geq n_0$. We know that $E(M,N) = \Ext^*_A(M, N)$ is finitely generated over ring of cohomology operators
$R = A[t_1,\ldots, t_c]$. Set
\[
E^{ev}(M, N) = \bigoplus_{n \geq 0}\Ext^{2n}(M, N) \quad \text{and} \quad E^{odd}(M, N) = \bigoplus_{n \geq 0}\Ext^{2n + 1}(M, N).
\]
Then $E(M, N) = E^{ev}(M, N)\oplus E^{odd}(M, N)$ as $R$-modules. We may assume that for some $m \geq n_0 $ the modules $E^{ev}(M, N)_{\geq m}$ and $E^{odd}(M, N)_{\geq m}$
are generated by $\Ext^{2m}_A(M,N)$ and $\Ext^{2m + 1}_A(M,N)$ respectively as $R$-modules.

Let $u_1,\ldots, u_r \in \Ext^{2m}_A(M,N)$ generate $E^{ev}(M, N)_{\geq m}$ as an $R$-module.
Note
\[
\Ext^{2m}_A(M,N)= \sHom_A(M, \Om^{-2m}(N)) \rt \Hom_{\Tc_X}(M, \Om^{-2m}(N)) = 0.
\]
So by \cite[2.1.26]{N} there exits $W_j$ with $\Vc^*(W_j)\subseteq X$ such that the map $u_j  \in \sHom_A(M, \Om^{-2m}(N))$ factors through $W_j$.
We have a commutative diagram
\[
\xymatrix{
\
&W_j
\ar@{<-}[dl]^{v_j}
\ar@{->}[dr]^{w_j}
 \\
M
\ar@{->}[rr]_{u_j}
&\
&\Om^{-2m}(N)
}
\]
So we have map $w_j^* \colon \Ext^*(M, W_j) \rt \Ext^*_{\geq 2m}(M, N)$ of $R$-modules. Note $u_j \in \image w_j^*$. Thus we have a surjective map
\[
\bigoplus_{j = 1}^{r} \Ext^*(M, W_j)  \xrightarrow{\sum w_i} E^{ev}(M, N)_{\geq m} \rt 0.
\]
Set $W = \oplus_j W_j$. Note $W \in \Sc_X$. Set $T = k[t_1,\ldots, t_c] = R\otimes k$. Tensoring with $k$ and taking annhilators (and their varieties) we get
\[
Var(\ann_T E^{ev}(M, N)_{\geq m}\otimes k) \subseteq X.
\]
Similarly we obtain
\[
Var(\ann_T E^{odd}(M, N)_{\geq m}\otimes k) \subseteq X.
\]
Notice
\[
\ann_T E^{odd}(M, N)_{\geq m}\otimes k) \cap  \ann_T E^{ev}(M, N)_{\geq m}\otimes k = \ann_T E(M, N)_{\geq 2m} = K (say).
\]
Clearly $Var(K) \subseteq X$. Set $\q$ to be  the unique graded maximal ideal
of
$T$. Then $\q^lK \subseteq \ann_T E(M, N)\otimes k$ for some $l \geq 1$. It follows that
\[
\Vc^*(M) \cap \Vc^*(N) \subseteq X.
\]
The result follows.
\end{proof}
\section{Proof of Theorem \ref{sym}}
In this section we give a proof of Theorem \ref{sym}. We first need:
\begin{lemma}\label{fart}[with hypotheses as in \ref{sym}]
Let $M \in \CMS(A)$ be essentially disjoint from $X$. Let $f \colon N \rt M$ and $g \colon M \rt L$ be in $\Mor_{\Sc_X}$. Then $N, L$ are essentially disjoint from $X$.
Further let $M = M_1 \oplus M_2$, $N = N_1 \oplus N_2$ and $L = L_1 \oplus L_2$ where $\Vc^*(M_1 \oplus N_1 \oplus L_1) \subseteq X$ and $\Vc^*(M_2 \oplus N_2 \oplus L_2) \subseteq \Pb^{c-1} \setminus X$. Consider the natural maps $i_M \colon M_2 \rt M$, $\pi_M \colon M \rt M_2$, $i_N \colon N_2 \rt N$ and $\pi_L \colon L \rt L_2$.
Then $\pi_M \circ f \circ i_N \colon N_2 \rt M_2$ and $\pi_L \circ g \circ i_M \colon M_2 \rt L_2$ are isomorphisms.
\end{lemma}
\begin{proof}
Note that $i_M, \pi_M \in \Mor_{\Sc_X}$. So $\pi_M \circ f \colon N \rt M_2$ and $g\circ i_M \colon M_2 \rt L$ are in $\Mor_{\Sc_X}$.
By \ref{mor-X} we get that
\[
\Vc^*(N) \setminus X = \Vc^*(M_2) \setminus X = \Vc^*(M_2).
\]
So we have a disjoint union
$$\Vc^*(N) = (\Vc^*(N)\cap X) \sqcup \Vc^*(M_2). $$
By \cite[3.1]{B} we get that $N = N_1 \oplus N_2$ where $\Vc^*(N_1) = \Vc^*(N)\cap X$ and $\Vc^*(N_2) =  \Vc^*(M_2)$. A similar assertion holds for $L$.

As $i_N \colon N_2 \rt N $ is in $\Mor_{\Sc_X}$ we get that  $h = \pi_M \circ f \circ i_N \colon N_2 \rt M_2$ is in $\Mor_{\Sc_X}$. So we have a triangle
\[
W \xrightarrow{\phi} N_2 \xrightarrow{h} M \xrightarrow{\phi} \Om^{-1}(W);
\]
where $W \in \Sc_X$.  As
$$\Vc^*(M_2)\cap X = \Vc^*(N_2)\cap X = \emptyset;$$
we get $\phi = 0$ and $\psi = 0$. As $\psi = 0$ we get by \cite[1.4]{Happel} that $\phi$ is a section. But $\phi = 0$. So $W= 0$. Thus $h$ is an isomorphism.
An analogous proof shows that $\pi_L \circ g \circ i_M $ is an isomorphism.
\end{proof}
We also need:
\begin{proposition}
\label{khoya}[with hypotheses as in \ref{sym}]
Let $M \in \CMS(A)$ be such that \\ $\Vc^*(M) \subseteq \Pb^{c-1} \setminus X$. Let $N, L \in \CMS(A)$ be any. Then the natural maps
\[
\Psi \colon \sHom_A(N, M) \rt \Hom_{\Tc_X}(N, M) \quad \text{and} \quad \Phi \colon \sHom_A( M, L) \rt \Hom_{\Tc_X}(M, L);
\]
are isomorphisms.
\end{proposition}
\begin{proof}
Suppose $\Psi(f) = 0$. Then factors through some $W \in \Sc_X$. So we have a commutative diagram
\[
\xymatrix{
\
&W
\ar@{<-}[dl]^{v}
\ar@{->}[dr]^{w}
 \\
N
\ar@{->}[rr]_{f}
&\
&M
}
\]
As $\Vc^*(M) \subseteq \Pb^{c-1} \setminus X$ we get that $w = 0$. So $f = 0$. Thus $\Psi$ is injective. A similar argument shows that $\Phi$ is injective.

Let $\xi \in  \Hom_{\Tc_X}(M, L)$. We write $\xi$ as a left fraction:
\[
\xymatrix{
\
&U
\ar@{->}[dl]_{u}
\ar@{->}[dr]^{f}
 \\
M
\ar@{->}[rr]_{\xi = fu^{-1}}
&\
&L
}
\]
where $u \in \Mor_{\Sc_X}$.  By \ref{fart}; $U$ is essentially disjoint from $X$. Say $U = U_1 \oplus U_2$ with $\Vc^*(U_1) \subseteq  X$ and
$\Vc^*(U_2) \subseteq \Pb^{c-1} \setminus X$. Let $i_U \colon U_2 \rt U$ be the natural map. Then $i_U \in \Mor_{\Sc_X}$. Set $u^\prime = u \circ i_U$ and
$f^\prime = f \circ i_U$. Then notice $\xi = f^\prime (u^\prime)^{-1}$. So we can write $\xi$ as an equivalent left fraction:
\[
\xymatrix{
\
&U_2
\ar@{->}[dl]_{u^\prime}
\ar@{->}[dr]^{f^\prime}
 \\
M
\ar@{->}[rr]_{\xi = f^\prime (u^\prime)^{-1}  }
&\
&L
}
\]
By \ref{fart} we get that $u^\prime$ is an isomorphism in $\CMS(A)$. It follows that $\xi \in \image \Phi$. Thus $\Phi$ is surjective.
The argument to show $\Psi$ is surjective is similar. However to prove this it is convenient to express elements of $\Hom_{\Tc_X}(N, M)$ as right fractions.
\end{proof}

Note that we proved half of Theorem \ref{sym} in the previous section. Here we prove a stronger result which implies the rest of Theorem \ref{sym} that we wish to prove.

\begin{theorem}
   \label{sym-part2}
(with hypotheses as in \ref{hypothesis} and notation as in \ref{defnX}.) Let $X$ be any algebraic set in $\Pb^{c-1}(k)$ . Let $M, N$ be any two MCM $A$-modules.
Assume $M$ or $N$ is essentially disjoint from $X$.
If $\Vc_X(M) \cap \Vc_X(N) = \emptyset$ then $\Hom_{\Tc_X}(M, \Om^{-n}(N)) = 0 \ \text{for all} \ n \in \Z$.
 \end{theorem}
\begin{proof}
Fix $n \in \Z$.
First assume that $M$ is essentially disjoint from $X$. Then $M = M_1 \oplus M_2$ where $\Vc^*(M_1) \subseteq X$ and $\Vc^*(M_2) \subseteq \Pb^{c-1} \setminus X$.
It follows that $\Vc^*(M_2)\cap \Vc^*(N) = \emptyset$. It follows that
$\sHom_A(M_2, N) = 0$. Notice $M_1 = 0$ in $\Tc_X$. So we have,
\begin{align*}
  \Hom_{\Tc_X}(M, \Om^{-n}(N) ) &= \Hom_{\Tc_X}(M_2, \Om^{-n}(N)) \\
   &\cong \sHom_A( M_2 , \Om^{-n}(N)), \quad \text{by \ref{khoya}};  \\
   &= 0
\end{align*}

The case when $N$ is essentially disjoint from $X$ is similar.

\end{proof}

\section{proof of Theorem \ref{murthy}}
In this section we give a proof of Theorem \ref{murthy}. We first need the following result Theorem \cite[3.7]{P}:
\begin{lemma}
  \label{mine}(with hypotheses as in \ref{hypothesis})
  Let $M \in \CMa(A)$ with $\cx_A M \geq 2$. Then there exists for some $n \geq 0$ a short exact sequence
  \[
  0 \rt K \rt \Om^{n+2}(M) \rt \Om^n(M) \rt 0
  \]
  where $\cx_A K = \cx_A M -1$.
\end{lemma}

We will also need the following:
\begin{definition}
Let $M \in \Tc_X$. Set
\[
\ecx M = \min \left\{ \cx_A L \mid L \cong M \ \text{in} \ \Tc_X \right \}.
\]
\end{definition}
The proof of Theorem \ref{murthy} is by induction.
It is convenient to prove a stronger form of the result in a special case:
\begin{theorem}
  \label{murthy-strong-empty}
  (with hypotheses as in \ref{hypothesis} and notation as in \ref{defnX}.)  Let $M, N$ be objects in $\Tc_X$.
  Assume $\Vc_X(M)\cap \Vc_X(N) = \emptyset$.
  Let $r = \ecx N $. Then
  \begin{align*}
    \Hom_{\Tc_X}(M, \Om^{-n} (N)) &= 0 \ \text{for } \ n = m, \cdots ,m+r-2   \\
    &\implies  \Hom_{\Tc_X}(M, \Om^{-n} (N) ) = 0 \ \text{for all }  \  n \geq m.
  \end{align*}
\end{theorem}
\begin{remark}
  Our Conjecture \ref{sym-conj} asserts that in the case above \\ $\Hom_{\Tc_X}(M, \Om^{-n} (N) ) = 0 $ for $n \gg 0$ without any initial condition.
\end{remark}
We now give:
\begin{proof}
[Proof of Theorem \ref{murthy-strong-empty}]  we use induction on $r = \ecx_A N$ to prove the result. We may assume $r = \cx_A N$.
We first consider the case when $r = 0$. Then $N = 0$  and so we have nothing to prove.
Next we consider the case when $r = 1$. Then note that $\Vc^*(N) = \{ a \}$ for some $a \in \Pb^{c-1} \setminus X$. The result follows from \ref{sym-part2}.

Next we consider the case when $r = \cx_A N \geq 2$ and assume the results are proved for MCM modules $E$ with $\ecx_A E < r$.
By \ref{mine}   we have a short exact sequence,
\[
0 \rt K \rt \Om^{n+2}(N) \rt \Om^n(N) \rt 0;
\]
for some $n \geq 0$ with $\cx_A K = r - 1$. So we have a triangle in $\CMS(A)$
\[
K \rt  \Om^{n+2}(N) \rt \Om^n(N) \rt \Om^{-1}(K)
\]
Rotating it several times we get the following triangle in $\CMS(A)$
\[
\Om^2(N) \rt N \rt E \rt \Om(N)
\]
where $\cx E = r -1$.  Also note $\Vc^*(E) \subseteq \Vc^*(N)$. We takes its image in $\Tc_X$
\[
s \colon   \Om^2(N) \rt N \rt E \rt \Om(N).
\]
Note $\Vc_X(M) \cap \Vc_X(E) = \emptyset$. By applying $\Hom_{\Tc_X}(M, -)$ to $s$ we obtain long exact sequence:
\begin{align*}
  \Hom_{\Tc_X}(M, \Om^{-j} (N)) &\rt \Hom_{\Tc_X}(M, \Om^{-j} (E)) \rt \Hom_{\Tc_X}(M, \Om^{-j + 1} (N)) \rt \tag{$\dagger$} \\
  \Hom_{\Tc_X}(M, \Om^{-j -1} (N)) &\rt \Hom_{\Tc_X}(M, \Om^{-j -1} (E)) \rt \Hom_{\Tc_X}(M, \Om^{-j} (N)) \rt.
\end{align*}
It follows that
\[
\Hom_{\Tc_X}(M, \Om^{-j} (E)) = 0 \quad \text{for} \  j = m+1,\ldots, m+ r-2.
\]
As $\ecx_A E \leq \cx_A E = r -1$ we get by induction hypotheses that \\
$\Hom_{\Tc_X}(M, \Om^{-j} (E)) = 0 $ for all $j \geq m+1$.
Putting $j = m + r -2$ in ($\dagger$) we get $\Hom_{\Tc_X}(M, \Om^{-m-r+1}(N)) = 0$.
Iterating we obtain $\Hom_{\Tc_X}(M, \Om^{-j} (N)) = 0$ for all $j \geq m$.
\end{proof}
We now state and prove a result which implies Theorem \ref{murthy}.
\begin{theorem}
  \label{murthy-strong}
  (with hypotheses as in \ref{hypothesis} and notation as in \ref{defnX}.)  Let $M, N$ be objects in $\Tc_X$.
  Let $r = \ecx N $. Then
  \begin{align*}
    \Hom_{\Tc_X}(M, \Om^{-n} (N)) &= 0 \ \text{for } \ n = m, m+1, \cdots. m+r   \\
    &\implies  \Hom_{\Tc_X}(M, \Om^{-n} (N) ) = 0 \ \text{for all }  \  n \geq m.
  \end{align*}
\end{theorem}
\begin{proof}
Set $a = \dim \Vc_X(M) \cap \Vc_X(N)$. We prove the result by induction on $a$.
We first consider the case when $a = -1$; i.e., $ \Vc_X(M) \cap \Vc_X(N) = \emptyset$.
In this case the result follows from Theorem \ref{murthy-strong-empty}.

Let $a = t\geq 0$ and assume the result is proved when $a < t$.
Next we consider the case when $r = \cx_A N \geq 2$ and assume the results are proved for MCM modules $E$ with $\ecx_A E < r$.
By \ref{mine}   we have a short exact sequence,
\[
0 \rt K \rt \Om^{n+2}(N) \rt \Om^n(N) \rt 0;
\]
for some $n \geq 0$ with $\cx_A K = r - 1$. So we have a triangle in $\CMS(A)$
\[
K \rt  \Om^{n+2}(N) \rt \Om^n(N) \rt \Om^{-1}(K)
\]
Rotating it several times we get the following triangle in $\CMS(A)$
\[
\Om^2(N) \rt N \rt E \rt \Om(N)
\]
where $\cx E = r -1$.  Also note $\Vc^*(E) \subseteq \Vc^*(N)$. We takes its image in $\Tc_X$
\[
s \colon   \Om^2(N) \rt N \rt E \rt \Om(N).
\]
Note $\Vc_X(M) \cap \Vc_X(E) \subseteq \Vc_X(M) \cap \Vc_X(E)$.
So $\dim \Vc_X(M) \cap \Vc_X(E) \leq t$.
By applying $\Hom_{\Tc_X}(M, -)$ to $s$ we obtain long exact sequence:
\begin{align*}
  \Hom_{\Tc_X}(M, \Om^{-j} (N)) &\rt \Hom_{\Tc_X}(M, \Om^{-j} (E)) \rt \Hom_{\Tc_X}(M, \Om^{-j + 1} (N)) \rt \tag{$\dagger \dagger$} \\
  \Hom_{\Tc_X}(M, \Om^{-j -1} (N)) &\rt \Hom_{\Tc_X}(M, \Om^{-j -1} (E)) \rt \Hom_{\Tc_X}(M, \Om^{-j} (N)) \rt.
\end{align*}
It follows that
\[
\Hom_{\Tc_X}(M, \Om^{-j} (E)) = 0 \quad \text{for} \  j = m+1,\ldots, r.
\]
As $\ecx_A E \leq \cx_A E = r -1$ we get by induction hypotheses that \\
$\Hom_{\Tc_X}(M, \Om^{-j} (E)) = 0 $ for all $j \geq m+1$.
Putting $j = m + r$ in ($\dagger \dagger$) we get $\Hom_{\Tc_X}(M, \Om^{-m-r-1}(N)) = 0$.
Iterating we obtain $\Hom_{\Tc_X}(M, \Om^{-j} (N)) = 0$ for all $j \geq m$.
\end{proof}

\section{ Proof of Theorems \ref{GARC}, \ref{murthyC} and \ref{HW}}
In this section we give proofs of Theorems \ref{GARC}, \ref{murthyC} and \ref{HW}. We first prove:
\begin{theorem}
\label{convo}
(with hypotheses as in \ref{hypothesis} and notation as in \ref{defnC}.)
 Fix $i \in \{1,\ldots, c-1 \}$.
Let $M, N$ be any two MCM $A$-modules.
Consider the following two statements:
\begin{enumerate}[\rm (1)]
  \item $\Hom_{\Tc_i}(M, \Om^{-n}(N)) = 0 \ \text{for all} \ n \gg 0$.
  \item $\dim (\Vc_i(M)\cap\Vc_i(N)) \leq i -1$.
\end{enumerate}
Then $(1) \implies (2)$.
\end{theorem}
\begin{proof}
  The proof of this result is completely analogous to proof of Theorem \ref{sym-main}. One simply replaces $W_j \in \Sc_X$ with $W_j \in \CMS_{\leq i}(A)$.
\end{proof}
As a consequence we give
\begin{proof}
[Proof of Theorem \ref{GARC}]
We get $\dim \topv(M) = \dim \Vc_i(M)   \leq i-1$.  This forces $\dim \Vc^*(M) \leq i-1$. So $\cx_A M \leq i$. Therefore $M = 0$ in $\Tc_i$.
\end{proof}
As we observed in the proof of Theorem \ref{murthy} it is convenient to prove a slightly more general result:
\begin{theorem}
  \label{murthyC-strong}
  (with hypotheses as in \ref{hypothesis} and notation as in \ref{defnC}.)  Let $M, N$ be objects in $\Tc_X$.
  Let $r = \cx N $. Then
  \begin{align*}
    \Hom_{\Tc_X}(M, \Om^{-n} (N)) &= 0 \ \text{for } \ n = m, m+1, \cdots. m+r - i   \\
    &\implies  \Hom_{\Tc_X}(M, \Om^{-n} (N) ) = 0 \ \text{for all }  \  n \geq m.
  \end{align*}
\end{theorem}
\begin{proof}
If $r \leq i$ then $N = 0$ in $\Tc_i$ and so we have nothing to prove.
Next we consider the case when $\cx_A N = i + 1$. By \ref{mine}
we have a short exact sequence:
\[
0 \rt K \rt \Om^{n+2}(N) \rt \Om^n(N) \rt 0;
\]
for some $n \geq 0$ and $\cx_A K = i$.
This yields an exact triangle in $\CMS(A)$
\[
K \rt \Om^{n+2}(N) \rt \Om^n(N) \rt \Om^{-1}(K).
\]
Taking the image of the above exact triangle in $\Tc_i$ and noting $K = 0$ in $\Tc_i$ we get
$\Om^{n+2}(N) \cong \Om^n(N)$  in $\Tc_i$. It follows that $N$ is two periodic in $\Tc_i$. The result trivially follows.

Rest of the proof is similar to proof of Theorem \ref{murthy-strong}.
\end{proof}

Finally we give:
\begin{proof}[Proof of Theorem \ref{HW}]
Suppose $U, V$ are MCM $A$-modules such that
\[
\Hom_{\Tc_{c-1}}(U, V) = \Hom_{\Tc_{c-1}}(U, \Om^{-1}(V)) = 0.
\]
Suppose if possible $U, V \neq 0$ in $\Tc_{c-1}$.
Thus $\cx_A U = \cx_A V = c$.
We have $\Vc^*(U) = \Vc^*(V) = \Pb^{c-1}$.
Thus $\Vc_{c-1}(U) = \Vc_{c-1}(V) = \Pb^{c-1}$.

Note $V$ is two periodic in $\Tc_{c-1}$. So we have
\[
\Hom_{\Tc_{c-1}}(U, \Om^{-j}(V)) = 0 \quad \text{for all} \ j \geq 0.
\]
By Theorem \ref{convo} we get $\dim (\Vc_{i-1}(U) \cap \Vc_{c-1}(V)) \leq c-2$. We get $\dim \Pb^{c-1} \leq c-2$ which is a contradiction.
\end{proof}

\end{document}